\documentclass[12pt]{amsart}
\usepackage{tikz-cd,enumerate,ulem}
\usepackage{cancel}
\usepackage{mathtools}
\usetikzlibrary{arrows}
\usepackage[utf8]{inputenc}
\usepackage{xspace}
\usepackage[dvipsnames]{xcolor}
\usepackage{soul}
\usepackage{comment}

\usepackage[english]{babel}
\usepackage[T1]{fontenc}
\usepackage{amsmath,amssymb}

\usepackage[a4paper,top=3cm,bottom=2cm,left=3cm,right=3cm,marginparwidth=1.75cm]{geometry}

\usepackage{amsthm}
\usepackage{amsfonts}
\usepackage{graphicx}

\usepackage[colorlinks=true, allcolors=blue]{hyperref}
\usepackage{thmtools}
\usepackage{cleveref}
\usepackage{url}
\newtheorem{theorem}{Theorem}[section]
\newtheorem{thm}[theorem]{Theorem}

\usepackage{crossreftools}

\newtheorem{proposition}[theorem]{Proposition}
\newtheorem{definition}[theorem]{Definition}
\newtheorem{lemma}[theorem]{Lemma}
\newtheorem{conjecture}[theorem]{Conjecture}

\newtheorem{cor}[theorem]{Corollary}

\theoremstyle{definition}

\newtheorem{remark}[theorem]{Remark}

{
      \theoremstyle{plain}
      
  }
  {
      \theoremstyle{plain}
      
  }

\def\EE{\mathbb{E}}
\def\FF{\mathbb{F}}
\def\GG{\mathbb{G}}
\def\KK{\mathbb{K}}
\def\LL{\mathbb{L}}
\def\MM{\mathbb{M}}

\def\QQ{\mathbb{Q}}

\def\ZZ{\mathbb{Z}}

\def\Qalg{\overline{\QQ}}

\def\GL{\mathrm{GL}}

\def\Qbar{\overline{\QQ}}
\DeclareMathOperator{\Hom}{Hom}
\DeclareMathOperator{\Gal}{Gal}

\newcommand{\mug}{{\boldsymbol\mu}}

\usepackage{enumitem}
\setlist[itemize,1]{label={--\,}}
\setlist{leftmargin=8mm}

\newtheorem*{fact*}{Fact} 
\newcommand{\Z}{{\mathbb{Z}}}

\newcommand{\Q}{{\mathbb{Q}}}

\newcommand\B{{\rm (B)}\xspace}

\newcommand{\into}{\hookrightarrow}
\newcommand{\onto}{\twoheadrightarrow}
\newcommand{\ovl}{\overline}
\newcommand{\wtl}{\widetilde}

\newcommand{\ccirc}{\kern0.5ex\vcenter{\hbox{$\scriptstyle\circ$}}\kern0.5ex}

\newtheorem*{theorem*}{Theorem}
\newtheorem*{remark*}{Remark}
\numberwithin{equation}{section}
\author{Andrea Conti, Ilaria Del Corso, Arnaud Plessis, Lea Terracini}
\title{Small points in radical extensions of number fields}

\begin{document}
\begin{abstract}
We study small points in radical extensions of algebraic fields. Given an algebraic extension $\FF$ of $\mathbb Q$, a finitely generated subgroup $\Gamma\subseteq \FF^\times$, and a rational prime $p$, we give a general criterion ensuring that
\(
\FF(\Gamma^{p\text{-}\mathrm{div}})\setminus \Gamma^{\mathrm{div}}
\)
has the Bogomolov property. This problem is motivated by a conjecture of Rémond, formulated when $\FF$ is a number field, predicting that such radical extensions contain no unexpected small points outside the divisible hull of the group used to generate them.
As applications, we obtain new cases of Rémond’s conjecture for radical extensions generated by division points with respect to a finite set of primes, recovering and extending previous results of Amoroso and the third author. Our argument is based on  a recent result on small points in $p$-adic Lie extensions.
\end{abstract}
\maketitle

\section*{Introduction} 

Given an algebraic number $\alpha$, we denote by $h(\alpha)$ its (absolute, logarithmic) Weil height. 
A set $\mathcal{Y}$ of algebraic numbers is said to have the \textit{Bogomolov property}, shortened as ‘property \B', if the Weil height of any element in $\mathcal{Y}$ is either $0$ or bounded from below by an absolute positive constant. 
This property is non-trivial since one easily sees that $\overline{\QQ}$ does not satisfy the property \B: a standard example of a sequence of algebraic numbers whose Weil height converges to 0 is given by $(b^{1/n})_n$ for any $b\in\Q^\times\setminus\{\pm 1\}$.

The interest in the Bogomolov property lies in the famous Lehmer's problem, which predicts that the Weil height of any algebraic number $\alpha$ of degree $D_\alpha$ is bounded from below by $CD_\alpha^{-1}$ for some positive absolute constant $C$, unless $\alpha$ is either $0$ or a root of unity. 
Of course, this conjecture holds true if we force $\alpha$ to belong to an algebraic set with property \B. 
This means that we can turn our attention to understanding algebraic numbers with arbitrarily small, but positive, Weil height, usually referred to as ‘small points', such as the terms in the sequence from the previous paragraph. 
The Galois orbit of a small point has a nice behavior since it is equidistributed around the unit circle according to Bilu's celebrated equidistribution theorem \cite{Bilu1997}. 
However, it does not tell us how to construct them \textit{explicitly}. 

Although the notion of property \B was formalized for the very first time in 2001 by E. Bombieri and U. Zannier \cite{BombieriZannier2001}, the explicit search for small points lying in a given field started much earlier. 
In that direction, a theorem of D. Northcott asserts that every number field has the property \B \cite{Northcott1950}. 
The first infinite algebraic extension of $\Q$ for which property \B was proved was the maximal totally real field extension of $\QQ$, by the work of A. Schinzel \cite{Schinzel1974}.
Since the publication of \cite{BombieriZannier2001}, the property \B has been established for a variety of fields: see for instance \cite{AmorosoDvornicich2000,AmorosoZannier2000,Habegger2013,AmorosoDavidZannier2014,FiliMilner2015,Grizzard2015,Galateau2016,Plessis2019,Frey2021,Frey2022,Plessis2024b,DixitKala2024,Sahu2025,AmorosoTerracini2026}. 

On the other hand, it is much more challenging to explicitly identify small points in a given field if the latter does not satisfy the property \B. In fact, as far as the authors know, there are only two results in the literature in this direction, the first one being a special case of the second. We briefly sketch them and explain how they lead to the work of the present paper.

Let $\KK$ be a number field, and let $\alpha\in\KK^\times\setminus\mug$, where $\mug$ denotes the group of all roots of unity. 
By some basic properties of the Weil height, $h(\zeta \alpha^{1/n})=h(\alpha)/n$ for every positive integer $n$ and every root of unity $\zeta$. 
The field $\KK(\mug, \alpha^{1/2}, \alpha^{1/3},\alpha^{1/4},\dots)$ therefore does not have the property \B. 
But does this field also contain non-obvious algebraic numbers of small height? In other words, does the set 
\[ \KK(\mug, \alpha^{1/2}, \alpha^{1/3},\alpha^{1/4},\dots) \setminus \left\{ \zeta \alpha^q, \zeta\in \mug, q\in\QQ \right\}\] 
have the property \B? 
A special case of a conjecture due to G. R\'emond predicts an affirmative answer to the last question, as we explain below.

Given a rational prime $p$, denote by $\mug_{p^\infty}$ the set of $p$-power roots of unity. 
The first known partial result towards the last question is due to F. Amoroso who showed that the set 
\[ \KK(\mug_{p^\infty}, \alpha^{1/p}, \alpha^{1/p^2}, \alpha^{1/p^3}, \dots) \setminus \left\{ \zeta \alpha^q, \zeta\in \mug, q\in\QQ \right\}\] 
has the property \B when $\KK=\QQ,\, \alpha\in \Z$ and $p$ is an odd rational prime not satisfying the Wieferich condition (that is, $p$ does not divide $\alpha$ and $p^2$ does not divide $\alpha^{p-1}-1$) \cite{Amoroso2016}. 
This result was then generalized by the third author who proved the same conclusion when $\KK$ is any number field, $\alpha$ is any element of $\KK$, and $p$ is any odd rational prime not dividing the discriminant of $\KK$ \cite{Plessis2022}. 

The aforementioned conjecture of Rémond sets the above results in a larger perspective. Let $\tilde\Gamma$ be a subgroup of a divisible abelian group $G$, written multiplicatively. 
We define the rank of $\tilde\Gamma$ as the maximal number of linearly independent elements of $\tilde\Gamma$, i.e. the dimension of the $\QQ$-vector space $\tilde\Gamma\otimes_\ZZ\QQ$. 
For any set of rational primes $S$, the $S$-division group of $\tilde\Gamma$ is the set of all $\alpha\in G$ for which there is a positive integer $n$, whose prime factors lie in $S$, satisfying $\alpha^n \in \tilde\Gamma$. 
We denote this set by $\tilde\Gamma^{S- \mathrm{div}}$. 
To ease notation, we omit the dependence on $S$ if the latter is the set of all rational primes. 
Note that $\tilde\Gamma$ and $\tilde\Gamma^{\mathrm{div}}$ have the same rank. 
The following is a very special case of \cite[Conjecture 3.4]{Remond2017}. 

\begin{conjecture}[R\'emond] \label{conj:Remond}
Let $\KK$ be a number field, and let $\tilde\Gamma \subseteq \overline{\QQ}^\times$ be a subgroup of finite rank. 
Then $\KK(\tilde\Gamma)\setminus \tilde\Gamma^{\mathrm{div}}$ has the property \B. 
\end{conjecture}

Conjecture \ref{conj:Remond} is equivalent to stating that for every number field $\KK$ and every \emph{ finitely generated subgroup} $\Gamma\subseteq \KK^\times$, the set $\KK(\Gamma^{\mathrm{div}})\setminus \Gamma^{\mathrm{div}}$ has the property \B. 
In other words, Conjecture \ref{conj:Remond} predicts that the only way to construct small points in a radical extension of $\KK$  is to extract roots of elements in $\KK$. 

\begin{remark}
\rm{Conjecture \ref{conj:Remond} naturally extends to any semi-abelian variety $\GG$ that is the Cartesian product of a power of the multiplicative group and an abelian variety defined over a number field $\KK$, see \cite[Sections 1 and 5]{Plessis2022} for more details. This generalized conjecture implies that $\GG(\Gamma^{\mathrm{div}})/\Gamma^{\mathrm{div}}$ is a free abelian group for all finitely generated subgroups $\Gamma\subseteq \GG(\KK)$. 
This last fact was recently proved by S. Checcoli and G. Dill \cite{CheccoliDill2025}, giving a significant argument in favor of Conjecture \ref{conj:Remond}. }
\end{remark}

A consequence of a theorem of F. Amoroso and U. Zannier \cite{AmorosoZannier2000} implies that Conjecture \ref{conj:Remond} is true for groups $\tilde{\Gamma}$ of rank $0$ since then $\tilde{\Gamma} \subseteq \mug$ and $\tilde{\Gamma}^{\mathrm{div}} = \mug$. 
The aforementioned theorems of F. Amoroso and the third author treat the special case where $\tilde\Gamma=\Gamma^{p-\mathrm{div}}$ for a subgroup $\Gamma\subseteq\KK^\times$ generated by a single element $\alpha$ and for an odd rational prime $p$ not dividing the discriminant of the number field $\KK$.

The main aim of this paper is to provide a general tool which identifies
sufficient conditions on an algebraic extension \(\FF\) of \(\QQ\) ensuring
that the only small points in \(\FF(\Gamma^{p-\mathrm{div}})\) are those lying
in \(\Gamma^{\mathrm{div}}\).

\begin{theorem} (Theorem \ref{thm:basecase})\label{thm:introbasecase}
Let $\FF$ be 
an algebraic extension of $\QQ$, $\Gamma\subseteq \FF^\times$ a finitely generated subgroup,  and $p$ a rational prime. 
Assume that:
\begin{itemize}
\item[a)] $\FF$ is Galois over some number field; 
\item[b)] $\FF'\setminus\Gamma^{\mathrm{div}}$ has the property \B for every finite extension $\FF'/\FF$;
\item[c)] there is a field embedding $\iota: \overline{\Q}\hookrightarrow \overline{\Q_p}$ such that the topological closure of $\iota(\FF)$ in $\overline{\Q_p}$ is a discrete valuation field.  
\end{itemize}
Then $\FF(\Gamma^{p-\mathrm{div}})\setminus\Gamma^{\mathrm{div}}$ has the property \B.
\end{theorem}

\subsection*{Consequences of our main tool} We state in this subsection several consequences of Theorem \ref{thm:introbasecase}. 
An easy proof by induction allows us to actually state a more general result involving finitely many rational primes. 
For any rational prime $p$, we write $\overline{\ZZ_p}^\times$ for the unit group of the ring of integers of $\overline{\QQ_p}$.

\begin{theorem} \label{thm:intro}
Let $\FF$ be 
an algebraic extension of $\QQ$, $\Gamma\subseteq \FF^\times$ a finitely generated subgroup, and $S$ a finite set of rational primes. 
Assume that: 
\begin{enumerate} [label=\alph*)]
\item $\FF
$ is Galois over some number field; 
\item $\FF'\setminus \Gamma^{\mathrm{div}}$ has the property \B for every finite extension $\FF'/\FF$; 
\end{enumerate}
and that for every $p\in S$, there exists a field embedding $\iota_p:\Qbar\hookrightarrow \overline{\QQ_p}$ such that
\begin{enumerate}[label=\alph*),resume]
\item the topological closure of $\iota_p(\FF)$ in $\overline{\QQ_p}$  is a discrete valuation field; 
\item $\iota_p(\Gamma)\subseteq \overline{\ZZ_p}^\times$, except possibly for one prime $p\in S$. 
\end{enumerate}
Then $\FF'(\Gamma^{S-\mathrm{div}})\setminus \Gamma^{\mathrm{div}}$ has the property \B for every finite extension $\FF'/\FF$. 
\end{theorem}

\begin{proof}
By induction on the size of $S$. The theorem is obvious when $S$ is empty since then $\FF'(\Gamma^{S-\mathrm{div}})=\FF'$. 
We now assume that $S$ is a non-empty set. 
By assumption, $\FF$ is Galois over some number field $\KK$. Pick a finite extension $\LL/\FF$ and let us show that $\LL(\Gamma^{S-\mathrm{div}})\setminus \Gamma^{\mathrm{div}}$ has the property \B. Since the Galois closure of $\LL$ over $\KK$ is a finite extension of $\FF$, we can therefore freely assume that $\LL/\KK$ is Galois. 

Let us select a rational prime $p\in S$ as follows: if $S$ is a singleton, then $p$ is its unique element. Otherwise, $p$ is any rational prime for which $\iota_p(\Gamma)\subseteq \overline{\ZZ_p}^\times$. 
Our induction hypothesis claims that the set $\FF'(\Gamma^{S\setminus \{p\}-\mathrm{div}}) \setminus \Gamma^{\mathrm{div}}$ has the property \B for every finite extension $\FF'/\LL$ (recall that $\LL/\FF$ is a finite extension). 

Denote by $F$ and $L$ the topological closures of $\iota_p(\FF)$ and $\iota_p(\LL)$ in $\overline{\QQ_p}$, respectively.
The fact that $L/F$ is a finite extension, because $\LL/\FF$ is, implies that $L$ is a discrete valuation field. 
We now claim that $L(\Gamma^{S\setminus \{p\}-\mathrm{div}})$ is a discrete valuation field as well. 
It is clear if $S$ is a singleton since then $L(\Gamma^{S\setminus \{p\}-\mathrm{div}})=L$. 
Otherwise, because $\iota_p(\Gamma)\subseteq \overline{\ZZ_p}^\times$, Kummer theory then explains that $L(\Gamma^{S\setminus \{p\}-\mathrm{div}})$ is an unramified extension of $L$.

Since $\LL(\Gamma^{S-\mathrm{div}})=\LL(\Gamma^{S\setminus \{p\}-\mathrm{div}})(\Gamma^{p-\mathrm{div}})$, the theorem now arises from Theorem \ref{thm:introbasecase} applied to $\FF=\LL(\Gamma^{S\setminus\{p\}-\mathrm{div}})$ (it is Galois over $\KK(\Gamma)$ and any finite extension of $\LL(\Gamma^{S\setminus\{p\}-\mathrm{div}})$ 
 is  of the form $\FF'(\Gamma^{S\setminus\{p\}-\mathrm{div}})$ for some finite extension $\FF'/\LL$).
\end{proof}

\begin{remark} Since $\Gamma$ is finitely generated, then for all but finitely many rational primes $p$, one has $\iota(\Gamma)\subseteq \overline{\ZZ_p}^\times$ for all field embeddings $\iota: \overline{\QQ} \hookrightarrow \overline{\QQ_p}$.
This means that condition $d)$ of Theorem \ref{thm:intro} merely prevents $S$ from containing rational primes that are in a certain finite set depending  only on $\Gamma$.
\end{remark}

\begin{remark} We thank Damien Roy for bringing our attention to the following interpretation of Theorem \ref{thm:intro}: 
Lehmer's problem has a positive answer for every field $\FF'(\Gamma^{S-\mathrm{div}})$. Indeed, by our theorem, it suffices to show that the answer is positive for the set $\Gamma^{\mathrm{div}}\setminus \mug$.   
Given $x\in \Gamma^{\mathrm{div}}\setminus \mug$, there is a least integer $n>0$ for which $x^n\in \Q(\mug, \Gamma)$. 
We have $x^n\notin \Q(\mug, \Gamma)^p$ for every prime divisor $p$ of $n$, since otherwise we would get $x^{n/p}\in \Q(\mug, \Gamma)$, contradicting the minimality of $n$. 
By \cite[Chapter VI, Theorem 9.1] {LangAlgebra}, the polynomial $X^n-x^n$ is irreducible over $\Q(\mug, \Gamma)$. 
In particular, $n=[\Q(\mug, \Gamma,x):\Q(\mug, \Gamma)]$.

Clearly, $x^n \in \Q(\mug, \Gamma)^\times \setminus \mug$. 
By a special case of  \cite[Theorem 1.1]{AmorosoZannier2000}, $\QQ(\Gamma,\mug)$ has the property \B, being an abelian extension of a number field.
Thus, there is a positive constant $c_\Gamma$, depending only on $\Gamma$, such that $[\Q(x):\Q]h(x)\geq n h(x)= h(x^n)\geq c_\Gamma$.
\end{remark}

We now describe four situations to which our theorems apply. 
The first occurs when $\FF$ is a number field, as in \cite{Amoroso2016, Plessis2022}. 
In this case, any finite extension of $\FF$ has the property \B and the completion of $\FF$ with respect to any finite place of $\FF$ is a discrete valuation field.
Theorem \ref{thm:intro} now implies that $\FF(\Gamma^{S-\mathrm{div}})\setminus \Gamma^{\mathrm{div}}$ has the property \B, provided that for all but at most one prime $p\in S$, there exists a field embedding $\iota_p : \Qbar\hookrightarrow \overline{\QQ_p}$ such that $\iota_p(\Gamma)$ is a subgroup of $\overline{\ZZ_p}^\times$. 

 Although the field $\FF(\Gamma^{S-\mathrm{div}})$  does not contain all roots of unity, a more refined application of Theorem \ref{thm:intro} shows that they can be included.

\begin{cor} \label{cor:corwithallroots}
Let $\Gamma\subseteq \overline{\QQ}^\times$ be a finitely generated subgroup, and let $S$ be a finite set of rational primes.
Assume that for all but at most one prime $p\in S$, there is a field embedding $\iota_p : \Qbar\hookrightarrow \overline{\QQ_p}$ such that $\iota_p(\Gamma)\subseteq \overline{\ZZ_p}^\times$. 
Then the set $\KK(\mug, \Gamma^{S-\mathrm{div}})\setminus \Gamma^{\mathrm{div}}$ has the property \B for every finite extension $\KK/\QQ(\Gamma)$. 
\end{cor}

\begin{proof}
 Write $\mug'$ for the set of roots of unity whose order is coprime to every rational prime in $S$. 
 A celebrated result of F. Amoroso and U. Zannier states that all finite extensions of $\FF:=\KK(\mug')$ have the property \B \cite[Theorem 1.1]{AmorosoZannier2000}. 
 It is well-known that every place of $\KK$ extending a rational prime in $S$ is unramified in $\FF$. 
 It now follows from Theorem \ref{thm:intro} that the set $\FF(\Gamma^{S-\mathrm{div}})\setminus \Gamma^{\mathrm{div}}$ has the property \B. 
 Since $\mug_{p^\infty}$ is contained in $\Gamma^{S-\mathrm{div}}$ for all $p\in S$, we get $\FF(\Gamma^{S-\mathrm{div}})=\KK(\mug,\Gamma^{S-\mathrm{div}})$.
\end{proof}

In the case where $\Gamma$ is the group generated by two rational primes $p$ and $q$, and $S =\{p,q\}$, condition $d)$ of Theorem \ref{thm:intro} is not satisfied. 
It does not therefore allow us to determine whether the set $\QQ(\mug, p^{1/p^\infty}, p^{1/q^\infty}, q^{1/p^\infty}, q^{1/q^\infty}) \setminus \Gamma^{\mathrm{div}}$ has the property \B.  
However, Theorem \ref{thm:intro} still yields a partial answer in this situation. 

\begin{cor}
Let $\Gamma$ be the group generated by two rational primes $p$ and $q$. 
Then the set $\KK(\mug, q^{1/q^\infty}, p^{1/p^\infty}, q^{1/p^\infty}) \setminus \Gamma^{\mathrm{div}}$ has the property \B for every number field $\KK$.
\end{cor}

\begin{proof} 
 Let $\mug'$ denote the set of roots of unity of order coprime to $p$. 
 Corollary \ref{cor:corwithallroots} applied to the group generated by $q$ and to the set $\{q\}$ shows that $\KK(\mug', q^{1/q^\infty})\setminus \Gamma^{\mathrm{div}}$ has the property \B for all number fields $\KK$. 
Since every place of $\KK$ extending $p$ is unramified in $\FF:=\KK(\mug',  q^{1/q^\infty})$, Theorem \ref{thm:introbasecase} proves that $\FF(\mug_{p^\infty}, p^{1/p^\infty}, q^{1/p^\infty})\setminus \Gamma^{\mathrm{div}}$ has the property \B.
\end{proof} 

All these corollaries fit into R\'emond's \Cref{conj:Remond}. 
However, our result is general enough to go beyond that setting, as illustrated by the next corollary. 

Given a number field $\KK$ and a finite place $v$ of $\KK$, set $\KK^{tv}$ to be the maximal algebraic extension of $\KK$ in which $v$ is totally split.  

\begin{cor}
Let $\KK$ be a number field, $v$ a finite place of $\KK$ lying over some rational prime $p$, and $\Gamma \subseteq (\KK^{tv})^\times$ a finitely generated subgroup. 
Then $\KK^{tv}(\mug, \Gamma^{p-\mathrm{div}})\setminus \Gamma^{\mathrm{div}}$ has the property \B.
\end{cor}

\begin{proof}
Let $\mug'$ denote the set of roots of unity of order coprime to $p$. 
We first show that every finite extension of $\KK^{tv}(\mug')$ has the property \B. 
As already noticed above, this amounts to prove  that $\LL(\mug')$ has the property \B for all finite extensions $\LL/\KK^{tv}$.  Since $\KK^{tv}/\KK$ is a Galois extension, the Galois closure of $\LL$ over $\KK$ is a finite extension of $\KK^{tv}$; we can therefore assume that $\LL/\KK$ is Galois. 
The subfield of $\LL(\mug')$ fixed by the center of $\Gal(\LL(\mug')/\KK)$ is contained in $\LL$ since this center clearly contains $\Gal(\LL(\mug')/\LL)$.
It now arises from \cite[Theorem 1.2]{AmorosoDavidZannier2014} that $\LL(\mug')$ has the property \B. 

The place $v$ being unramified in both $\KK^{tv}$ and $\KK(\mug')$, and therefore in $\KK^{tv}(\mug')$, the corollary now follows by applying Theorem \ref{thm:introbasecase} to $\FF=\KK^{tv}(\mug')$. 
\end{proof}

\subsection*{Outline of the paper} Our approach to Rémond's conjecture differs from those of
\cite{Amoroso2016,Plessis2022,Plessis2024}. We describe it below. 

The two key ingredients are a lemma of R\'emond (\Cref{lem:Remstronggen}) and a recent result due to P. Piras together with the first and fourth author (\Cref{thm:Piras}); they are stated in Section \ref{sec:B}. More precisely, Theorem \ref{thm:Piras} describes the small points in certain Galois extensions that satisfy a set of assumptions: the base field is a discrete valuation field at a $p$-adic place (property (DVF)), the inertia subgroup over that place equals the decomposition subgroup (property (LTR)),
the Galois group is a $p$-adic Lie group of positive dimension (property (LIE)), and the normal closure, in the global Galois group, of a decomposition subgroup of a graded piece of the Lie filtration is the whole graded piece (property (NC)). 

Section~\ref{sec:preliminariesonradexts} collects some basic properties of radical extensions.
Sections~\ref{sec:NC} and~\ref{sec:LPTR} are devoted to proving properties \textup{(NC)}, (LIE)
and \textup{(LTR)} for radical extensions of the form
\(
  \FF(\Gamma^{p\text{-}\mathrm{div}})/\FF,
\)
where \(\FF\) is a subfield of \(\Qalg\) satisfying suitable assumptions; see Theorem~\ref{thm:NC}, Proposition \ref{prop:lie} and Corollary~\ref{cor LPTR}, respectively. 

We complete the proof of Theorem \ref{thm:introbasecase} in the final Section \ref{sec:mainproof}. 
A descent argument involving Theorem \ref{thm:Piras}, which will be applied to the extension $\FF(\Gamma^{p-\mathrm{div}})/\FF$, will show that $\FF(\Gamma^{p-\mathrm{div}})\setminus (\FF(\Gamma^{1/p^m})\cdot \Gamma^{\mathrm{div}})$ has the property \B for some integer $m\geq 0$. 
By assumption $b)$ of Theorem \ref{thm:introbasecase}, the set $\FF(\Gamma^{1/p^m}) \setminus \Gamma^{\mathrm{div}}$ has the property \B, and so, by Theorem \ref{lem:Remstronggen}, $(\FF(\Gamma^{1/p^m})\cdot\Gamma^{\mathrm{div}}) \setminus  \Gamma^{\mathrm{div}}$ too.
Gathering these facts will end the proof.

\subsection*{Acknowledgments} 
The first author is funded by the Deutsche Forschungsgemeinschaft - Project-ID 444845124 – TRR 326 GAUS. 
The second and fourth authors are members of the INDAM group GNSAGA.
The third author was supported by the Beijing Municipal Natural Science Foundation (No. IS25018).
The second author gratefully acknowledges the support of  the MIUR Excellence Department Project awarded to the Department of Mathematics, University of Pisa, CUP I57G22000700001.
The second author has performed this activity in the framework of the PRIN 2022, titled \emph{Semiabelian varieties, Galois representations and related Diophantine problems}.

\subsection*{Notation}\label{sec:notation}
We will use throughout the paper the following notation:
\begin{itemize}
\item $\mug$\,: group of all roots of unity in $\ovl\Q$; 
\item $\mug_n$ for a positive integer $n$\,: the subgroup of all $n$-th roots of unity;
\item  $\mug_{p^\infty}$  for a prime $p$\,: the subgroup of all $p$-power roots of unity;
\item $\zeta_n$ for a positive integer $n$\,: a primitive $n$-th root of unity;
\item $\KK,\LL, \FF, \dots$ for algebraic extensions of $\QQ$; 
\item $K,L,F,\dots$ for generic fields of characteristic 0; 
\item $K^\times$ is the multiplicative group of a field $K$;
\item $\overline{K}$ denotes an algebraic closure of the field $K$;
\item $[L:K]$\,: the degree of a field extension $L/K$; 
\item $\FF_p$, for a prime $p$\,: the field with $p$ elements;
\item $\langle S\rangle$\,: group generated by a subset $S$ of some group;
\item $G\cdot H$\,: the group $\{gh, g\in G, h\in H\}$, with $G$ and $H$ two abelian subgroups of a common abelian group;
\item if $K$ is a field, $\Gamma$ a subgroup of $K^\times$, and $n$ is a positive integer\,: $$\Gamma^{1/n}=\{\alpha\in\overline{K}^\times\ |\ \alpha^n\in\Gamma\};$$
\item a \emph{fixed} field embedding
\(
\overline{\QQ}\hookrightarrow \overline{\QQ_p}
\)
will be regarded as an inclusion: in this way, \(\overline{\QQ}\) will be viewed as a subfield of \(\overline{\QQ_p}\).
\end{itemize} 

\setcounter{tocdepth}{1}
\tableofcontents

\section{Preliminary results on the Bogomolov property}\label{sec:B}
Given an algebraic number $\alpha$ and a number field $\KK$ containing it, we define the \textit{(absolute logarithmic) Weil height} of $\alpha$ by
\begin{equation*} h(\alpha)=\frac 1{[\KK:\QQ]} \sum_{v} [\KK_v:\QQ_v] \log \max\{1,|\alpha|_v\},\end{equation*}
where $v$ varies over the set of places of $\KK$ and the nonarchimedean absolute values extend the usual $p$-adic ones, for every $p$. It is easily checked that such a definition depends only on $\alpha$, not on $\KK$. 

Following \cite{BombieriZannier2001}, we say that:

\begin{definition}
A set $\mathcal{Y}$ of algebraic numbers is said to have the \textup{Bogomolov property} (property \B for short) if there is a positive constant $C$ such that $h(\alpha)\geq C$ for every $\alpha\in\mathcal{Y}\setminus(\mug\cup\{0\})$. 
\end{definition}

For later use, we record a result of Rémond. 

\begin{thm}[R\'emond] \label{lem:Remstronggen}
Let $\Gamma\subseteq \overline{\QQ}^\times$ be a subgroup of finite rank, and let $\LL/ \Q(\Gamma)$ be an algebraic extension. 
Then the set $\LL^\times \setminus \Gamma^{\mathrm{div}}$ has the property \B if and only if the set $\LL^\times \cdot \Gamma^{\mathrm{div}} \setminus \Gamma^{\mathrm{div}}$ has the property \B. 
\end{thm}
 
\begin{proof}
The backward direction is trivial and the forward one comes from \cite[Corollary 2.3]{Pottmeyer2021}, since $\Gamma\subseteq \Gamma^{\mathrm{div}} \cap \LL^\times \subseteq \Gamma^{\mathrm{div}}$ (recall that $\Gamma$ and $\Gamma^{\mathrm{div}}$ have the same rank). 
\end{proof}

We also state here a recent result which will play a crucial role in the proof of our main result. 
\begin{theorem}[{\cite[Theorem 3.2]{ContiPirasTerracini2026}}]\label{thm:Piras}
Let $\KK$ be a number field, and consider a tower of algebraic extensions
$$\KK\subseteq \FF\subseteq \LL_0\subseteq \LL_1\subseteq \ldots \subseteq \LL=\bigcup_n \LL_n.$$
Choose a field embedding $\ovl\Q\into\overline{\Q_p}$ and write $K,F,L,L_i$ for the topological closures in $\overline{\Q_p}$ of the fields denoted by the corresponding bold letter. 
Assume that the following conditions are satisfied:
\begin{description}
\item[\rm{(GAL)}] $\LL_n/\KK$ is Galois for every $n$. 
\item[\rm{(LIE)}]\label{ass1:Lie} $\Gal(\LL/\FF)$ is a $p$-adic Lie group of positive dimension, and $(\Gal(\LL/\LL_n))_n$ is a Lie filtration of $\Gal(\LL/\FF)$. 
\item[\rm{(DVF)}]\label{ass1:inerzia finita} $F$ is a discrete valuation field. 
\item[\rm{(LTR)}]\label{ass1:LPTR} The local extension $L/L_0$ is totally ramified. 
\item[\rm{(NC)}] \label{ass1:NC} For every $n$, the normal closure of $\Gal(L_{n+1}/L_n)$ in $\Gal(\LL_{n+1}/\FF)$ is $\Gal(\LL_{n+1}/\LL_n)$. 
\end{description}
Then the set 
\[\mathcal{Z}\coloneqq  \bigcup_{n\ge 1}\{\alpha\in\LL_n \ |\ \exists \sigma\in \Gal(\LL/\LL_{n-1}) \hbox{ such that } \sigma(\alpha)/\alpha\not\in\mug_{p^\infty}\} \]
has the property \B.
\end{theorem}

\section{Preliminaries on radical extensions} \label{sec:preliminariesonradexts} 
We collect in this section the results on radical extensions that we will need in the following. 
We start with the main statement of Kummer theory. Our reference is \cite[Chapter~1, \S 5]{NeukirchCFT}.  

\begin{theorem}\label{thm:kummer}
 Let $n$ be a positive integer, and let $K$ be a field of characteristic $0$ containing $\mug_n$.
 Then the abelian extensions $L/K$ of exponent dividing $n$ are in one-to-one correspondence with the subgroups of $K^\times $ containing $K^{\times n}$, via the map $L\mapsto\Delta=L^{\times n}\cap K^\times$. \\
 With this notation, $L=K(\Delta^{1/n})$ and we have a canonical isomorphism
 $$\Hom(\Gal(L/K),\mug_n)\cong \Delta/{K^{\times n}}.$$
 \end{theorem}
 
\begin{remark} \label{rem:kummer}
With the notation of Theorem~\ref{thm:kummer}, if  $\Gamma\subseteq K^\times$ is a subgroup and $L=K(\Gamma^{1/n})$, then:
\begin{enumerate} 
\item $\Gamma \cdot K^{\times n}=\Delta=L^{\times n}\cap K^\times$; 
\item $\displaystyle{\Hom}({\Gal}(L/K),\mug_n)\cong \Delta/{K^{\times n}}\cong \Gamma /(\Gamma\cap K^{\times n}).$
\end{enumerate}
\end{remark}

\begin{lemma} \label{lemma:polyirr}
Let $p$ be a rational prime, and let $K$ be a field of characteristic $0$ containing $\mug_4$ if $p=2$. 
Then for all $a\in K$ and all positive integers $n$, the polynomial $X^{p^n}-a$ is irreducible over $K$ if and only if $X^p-a$ is irreducible over $K$.     
\end{lemma}

\begin{proof}
This is a direct consequence of \cite[Chapter VI, Theorem 9.1] {LangAlgebra} since we have $-4K^4=((1+\zeta_4)K)^4=K^4$.
\end{proof}

The following lemma will be used at the very end of the paper to complete our descent argument.

\begin{lemma} \label{lem:orbitagen2}
Let $p$ be a rational prime, $K$ a field of characteristic $0$ containing $\mug_{2p}$, and $\Gamma\subseteq K^\times$ a finitely generated subgroup. 
If $\delta\in K(\Gamma^{1/p})^\times$ satisfies $\sigma(\delta)/\delta\in\mug_{p^{\infty}}$ for every $\sigma\in \Gal(K(\Gamma^{1/p})/K)$, then $\delta^p\in\Gamma\cdot K^p$. 
\end{lemma}
\begin{proof} 
 Since $\Gal(K(\Gamma^{1/p})/K)$ is a finite group, there is a least integer $m$ such that $\sigma(\delta)/\delta\in\mug_{p^m}$ for every $\sigma\in \Gal(K(\Gamma^{1/p})/K)$. 
It is enough to show that $m\leq 1$, since then Galois theory gives $\delta^p\in K^\times \cap K(\Gamma^{1/p})^{\times p}=\Gamma\cdot K^{\times p}$, the equality coming from Remark \ref{rem:kummer}.

Let $\sigma\in\Gal(K(\Gamma^{1/p})/K)$ such that $\sigma(\delta)=\zeta\delta$, with $\zeta$ a primitive $p^m$-th root of unity, and set $\sigma(\zeta)=\zeta^a$. Clearly $a,p$ are coprime, and more precisely, $a\equiv 1 \pmod{p}$ if $p$ is odd and $a\equiv 1 \pmod{4}$ if $p=2$ since $\mug_{2p}\subseteq K$ by assumption.

Then $\sigma^p(\delta)=\zeta^{1+a+a^2+\ldots +a^{p-1}}\delta$.
 Since $\Gal(K(\Gamma^{1/p})/K)$ has exponent $p$,  it follows that  $\sigma^p$ is the identity, and this implies
$$1+a+a^2+\ldots +a^{p-1}\equiv 0\pmod{p^m}.$$
If $p=2$, then $a\equiv 1 \pmod{4}$ and $1+a\equiv 0 \pmod{2^m}$, which leads to $m\leq 1$. 

If $p$ is odd, then we put $b=a-1$ and write $$1+a+a^2+\ldots +a^{p-1}=\frac{(b+1)^p-1}{b}=\sum_{j=1}^{p} \binom p j b^{j-1}\equiv p \pmod{p^2}$$ since $p$ divides both $b$ and $\binom p 2$. It follows that $m\leq 1$ as well. 
\end{proof}
 
As usual, $\vert S\vert$ denotes the cardinality of a finite set $S$. 

\begin{lemma} \label{lem:stabdegree}
Let $p$ be a rational prime, $K$ a field of characteristic 0 containing $\mug_{2p}$, and $\Gamma\subseteq K^\times$ a finitely generated subgroup.
Then, $[K(\Gamma^{1/p}):K]=|\Gamma/\Gamma^p|$ if and only if $\Gamma\cap K^{\times p}=\Gamma^p$.
\end{lemma}
\begin{proof}
 By Remark \ref{rem:kummer}, the degree of $K(\Gamma^{1/p})/K$ is the order of the $\FF_p$-vector space $\Gamma/(\Gamma\cap K^{\times p})$. Since $\Gamma^p\subseteq \Gamma\cap K^{\times p}$, the equality $[K(\Gamma^{1/p}):K]=|\Gamma/\Gamma^p|$ holds if and only if the natural surjection 
$\Gamma/\Gamma^p\onto\Gamma/(\Gamma\cap K^{\times p})$ is an isomorphism, namely when $\Gamma\cap K^{\times p}=\Gamma^p$.  \end{proof}

\begin{lemma} \label{lmm:Mgaloisifroots}
Let $p$ be a rational prime, $K$ a field of characteristic 0 containing $\mug_{2p}$, and $\Gamma\subseteq K^\times$ a finitely generated subgroup. 
 Assume that  $[K(\Gamma^{1/p}):K]=|\Gamma/\Gamma^p|$.
Then, for all  integers $m\geq 0$  and all fields $M$ satisfying $K(\Gamma^{1/p^m})\subsetneq M \subseteq K(\Gamma^{1/p^{m+1}})$, the field $M$ is Galois over $K$ if and only if $M$ contains $\mug_{p^{m+1}}$. 
\end{lemma}
\begin{proof} 
Set $K_n=K(\Gamma^{1/p^n})$ for all non-negative integers $n$.
Since $\mug_p\subseteq K$, then $K_{m+1}/K_m$ is a Kummer extension  of exponent $p$, and the same holds for $M/K_m$. 
Thus, $M=K_m(\delta_1,\dots, \delta_t)$ for some positive integer $t$ and some $\delta_i\in K_{m+1}\setminus K_m$ such that $\delta_i^p\in K_m$. 
Remark \ref{rem:kummer} tells us that we can in fact choose $\delta_i$ in $\Gamma^{1/p^{m+1}}\setminus \Gamma^{1/p^m}$, or, in other words, $\delta_i^{p^{m+1}}\in \Gamma\setminus\Gamma^p$. Lemma \ref{lem:stabdegree} guarantees that under our assumptions, we have $\Gamma\cap K^{\times p}=\Gamma^p$; whence $\delta_i^{p^{m+1}}\notin K^{\times p}$. 
By Lemma \ref{lemma:polyirr}, the polynomial $X^{p^{m+1}}-\delta_i^{p^{m+1}}$ is irreducible over $K$ and the Galois orbit of $\delta_i$ over $K$ is therefore the set $\mug_{p^{m+1}}\cdot \delta_i$.

If $M/K$ is Galois, then necessarily $\mug_{p^{m+1}}\cdot \delta_i\subseteq M$, and therefore $\mug_{p^{m+1}}\subseteq M$.
Conversely, assume that $\mug_{p^{m+1}}\subseteq M$. The field $M$ is generated over $K$ by $\Gamma^{1/p^m}$ and $\delta_1,\dots, \delta_t$. It therefore contains the all the Galois orbits of its generators over $K$.
\end{proof}

\begin{remark}\label{rem:reduction-maximal-degree}
Let \(K\) be a field of characteristic \(0\) containing \(\mug_{2p}\), and let
\(\Gamma\subseteq K^\times\) be a finitely generated subgroup. Suppose that
\(\Gamma\) can be generated by \(r\) elements and that
\(
[K(\Gamma^{1/p}):K]=p^r
\).
Then
\(
[K(\Gamma^{1/p}):K]=|\Gamma/\Gamma^p|
\).
Indeed, by Remark~\ref{rem:kummer},
\[
p^r=[K(\Gamma^{1/p}):K]
=
|\Gamma/(\Gamma\cap K^{\times p})|
\leq |\Gamma/\Gamma^p|
\leq p^r.
\]
We shall use this observation in Proposition~\ref{prop:degradext} and
Theorem~\ref{thm:NC} in order to reduce to the situation of
Lemma~\ref{lem:stabdegree}. In both cases, after replacing \(K\) by a finite extension \(\widetilde K\)
and \(\Gamma\) by a finitely generated subgroup
\(\widetilde\Gamma\subseteq \widetilde K^\times\), generated by \(r\) elements,
we shall have
\(
[\widetilde K(\widetilde\Gamma^{1/p}):\widetilde K]=p^r
\).
\end{remark}

The following proposition shows that, after a finite level, a radical tower has a stable Kummer structure: the degrees grow by a factor $p^s$ at each step, and the higher layers are obtained by adjoining successive \(p\)-power roots of exactly $s$ fixed elements.

\begin{proposition}\label{prop:degradext}
Consider a rational prime $p$, a field $K$ of characteristic $0$, and  a finitely generated subgroup $\Gamma \subseteq K^\times$. For each integer $m\ge 1$, set
$$
K_m:=K(\Gamma^{1/p^m}).
$$
Then there exist integers $n\geq 2$ and $s\geq 0$ such that, for all $m\geq n$, we have
$$[K_m:K_n]=p^{s(m-n)}.$$ 
Moreover, if $s\ge 1$, then  there exist elements $\gamma_1,\dots,\gamma_s\in \Gamma^{1/p^n}$ such that, for all integers $m\geq n$, we have
$$K_m=K_n\left(\gamma_1^{1/p^{m-n}}, \dots, \gamma_s^{1/p^{m-n}}\right).$$ 
Finally, if $K$ contains $\mug_{2p}$, but not $\mug_{p^\infty}$, then  $s\geq 1$ and there is a permutation of $\gamma_1,\dots, \gamma_s$ such that \begin{equation}\label{eq:replace} K_m=K_n\left(\zeta_{p^{k_n}}^{1/p^{m-n}}, \gamma_2^{1/p^{m-n}}, \dots, \gamma_s^{1/p^{m-n}}\right)\end{equation} for all integers $m\geq n$, where $\mug_{p^{k_n}}= K_n\cap \mug_{p^\infty}$.
\end{proposition}

\begin{proof}
For each $m\ge 1$, let
$
\Gamma_m:=\Gamma^{1/{p^m}}\subseteq K_m^\times$
and
$V_m:=\Gamma_m/(\Gamma_m\cap K_m^{\times p}).$
Since $\Gamma_m$ is finitely generated, $V_m$ is an $\mathbb F_p$-vector space of finite dimension
$$
s_m:=\dim_{\mathbb F_p}V_m. 
$$
In addition,  the sequence $(s_m)_m$ is bounded.
As $K_{m+1}=K_m(\Gamma_m^{1/p})$ and $\mug_p\subseteq K_m$, we infer from  Remark \ref{rem:kummer} that 
$[K_{m+1}:K_m]$ equals the size of $V_m$, namely $p^{s_m}$, 
and 
 $\Gal(K_{m+1}/K_m)$ is an elementary abelian $p$-group; in particular, it has exponent $p$.

We claim that the sequence $(s_m)_{m\ge 2}$ is non-decreasing:
this will follow if we prove that there is a surjective map from $V_{m+1}$ onto $V_m$.

Let $\psi\colon\Gamma_{m+1}\to\Gamma_m$ be the group homomorphism defined by $x\mapsto x^p$. 
Denote by $\pi_m$ the projection of 
$\Gamma_m$ onto $V_m$ and let 
$$\varphi=\pi_m\circ\psi\colon\Gamma_{m+1}\to V_m.$$
We claim that this map induces a surjective homomorphism $\ovl\varphi\colon V_{m+1}\to V_m$.
Clearly, $\varphi$ is surjective, since both $\pi_m$ and $\psi$ are surjective. It remains to check that $\varphi$ descends to the quotient, namely that if $x\in\Gamma_{m+1}\cap K_{m+1}^{\times p}$, then $x\in\mathop{ker}\varphi$, the last meaning that $x^p\in K_m^{\times p}$.
Now, $x=y^p$ for some $y\in K_{m+1}^\times$, therefore $x^p=y^{p^2}\in\Gamma_m\subseteq K_m$. 
If $x^p$ is not a $p$-power in $K_m$, then, by Lemma \ref{lemma:polyirr}, $X^{p^2}-x^p$ is an irreducible polynomial in $K_m[X]$ that has a root, and therefore all its roots, in its Galois extension $K_{m+1}$. Now, since $m\ge2$ we have $\mug_{p^2}\subseteq K_m$, so  $K_m(y)/K_m$ is a cyclic extension of order $p^2$. 
But this cannot be the case since $K_m(y)\subseteq K_{m+1}$, which is elementary abelian over $K_m$. This shows that $x^p$ is a $p$-power in $K_m^\times$, and so $\varphi$ descends to a surjective map on the quotient, therefore, $s_{m+1}\ge s_m$, as claimed. 
Since the non-decreasing sequence $(s_m)_m$ is bounded, it eventually becomes constant. Let $s$ be its ultimate value and choose $n\ge 2$ such that $s_m=s$ for all $m\geq n$.

We have $[K_{m+1}:K_m]=p^s$ for all $m\geq n$ and the multiplicativity formula for degrees gives  $[K_m:K_n]= p^{s(m-n)}$ for all $m\ge n$, proving the first part of the proposition. 

We  now focus on the second part of the claim. Choose $\gamma_1,\dots,\gamma_s\in \Gamma_n$ whose classes modulo $\Gamma_n \cap K_n^{\times p}$ form a basis of $V_n$. 
Our previous argument shows that the classes of \[ \gamma_1^{1/p^{m-n}},\dots, \gamma_s^{1/p^{m-n}} \] form a basis of $V_m$.
Concretely, every element of $\Gamma_m$ decomposes as $z^p\prod_{i=1}^s \gamma_i^{a_i/p^{m-n}}$ for some $z\in K_m$ and some integers $a_1,\dots,a_s\in \{0,1,\dots,p-1\}$. 
We now obtain the second part of the statement thanks to an easy induction on $m\geq n$ since 
\begin{align*} \label{eq:rootswelldefined}
\begin{split}
 K_m & = K_{m-1}(\Gamma_{m-1}^{1/p})\\
 & = K_n\left(\gamma_1^{1/p^{\,m-1-n}},\dots,\gamma_s^{1/p^{\,m-1-n}}\right) \left(\gamma_1^{1/p^{\,m-n}},\dots,\gamma_s^{1/p^{\,m-n}}\right) \\ 
 & = K_n \left(\gamma_1^{1/p^{\,m-n}},\dots,\gamma_s^{1/p^{\,m-n}}\right),
 \end{split}
\end{align*}
where the second equality follows from the inductive hypothesis.

 We now show the last part of the proposition.  Recall that $K$ contains $\mug_{2p}$ but not $\mug_{p^\infty}$. We cannot have $s=0$, since otherwise we would get $K(\Gamma^{p-\mathrm{div}})=K_n$, contradicting the fact that $K_n$ does not contain $\mug_{p^\infty}$ as a finite extension of $K$. 
 
 Let $\tilde{\Gamma}\subseteq \Gamma_n$ denote the group generated by $\gamma_1,\dots,\gamma_s$.
For every integer $m\geq n$, we put $\mug_{p^{k_m}}=K_m \cap \mug_{p^\infty}$. 
Note that $k_{m+1}\leq k_m +1$ since $K_{m+1}/K_m$ is a Kummer extension of exponent $p$.

Let $X$ be the set of integers $j$ such that $j\geq n+k_n+1$ and $k_j=k_{j-1}+1$. Since $K(\Gamma^{p-\mathrm{div}})$ contains $\mug_{p^\infty}$, $X$ is an infinite set. Pick an element $j\in X$.

Of course, $k_j\leq k_n+j-n$, and so $\zeta_{p^{k_j-j+n}}=\zeta_{p^{k_j}}^{p^{j-n}}\in K_n$. 
Since moreover $\zeta_{p^{k_j}} \in K_j=K_n(\tilde\Gamma^{1/p^{j-n}})$, Remark \ref{rem:kummer} shows that  $\zeta_{p^{k_j}}^{p^{j-n}}\in\tilde\Gamma\cdot K_n(\mug_{p^{j-n}})^{\times{p^{j-n}}}$,  i.e. $\zeta_{p^{k_j}} = c\alpha^{1/p^{j-n}}$ with $c\in K_n(\mug_{p^{j-n}})\subseteq K_{j-1}$ (so that $K_n(c)/K_n$ is abelian) and $\alpha\in \tilde{\Gamma}$. 
Since $\zeta_{p^{k_j}}\notin K_{j-1}$, we infer that $\alpha\notin \tilde{\Gamma}^p$.
Raising to the $p^{j-n}$-th power, we get $\zeta_{p^{k_j-j+n}}=c^{p^{j-n}}\alpha$. 
In particular, $c^{p^{j-n}}\in K_n$ and a theorem of Schinzel \cite[Theorem 2]{Schinzel1977} about radical extensions with abelian Galois group then asserts that  $c^{p^{j-n+k_n}}\in  K_n^{p^{j-n}}$, that is $c^{p^{j-n}}\in \mug_{p^{k_n}} \cdot K_n^{p^{j-n-k_n}}$. 
We thus get $\zeta=d^{p^{j-n-k_n}}\alpha$ for some $\zeta\in \mug_{p^{k_n}}$ and some $d\in K_n$. 
We cannot have $\zeta\in \mug_{p^{k_n-1}}$, since then $\alpha$ would be in $K_n^p\cap\tilde{\Gamma}=\tilde{\Gamma}^p$ according to Lemma \ref{lem:stabdegree}, a contradiction. 
In conclusion, \begin{equation*}\label{eq:kkk} \zeta_{p^{k_n}}=d^{p^{j-n-k_n}} \prod_{i=1}^s \gamma_i^{a_i}\end{equation*} for some integers $a_1,\dots, a_s$, at least one of which is coprime to $p$, say $a_{i_j}$. 

By the pigeonhole principle, there is $i\in\{1,\dots,s\}$ such that $i_j=i$ for infinitely many $j\in X$.  
After permuting the elements $\gamma_1,\dots,\gamma_s$, we can assume that $i=1$. 

Now, let $m\geq n$  be an integer, and let $j\geq m+k_n$ be an element of $X$ with $i_j=1$. 
 Since $m-n\leq j-n-k_n$, B\'ezout's identity gives $\gamma_1^{1/p^{m-n}}\in K_n(\zeta_{p^{k_n}}^{1/p^{m-n}}, \gamma_2^{1/p^{m-n}},\dots,\gamma_s^{1/p^{m-n}})$. This proves \eqref{eq:replace}.
 \end{proof}

\begin{remark}\label{rem:km} Since the Galois group of the extension $K_m/K_n$ has exponent $p^{m-n}$, we immediately infer from \eqref{eq:replace} that $K_m\cap\mug_{p^\infty}=\mug_{p^{k_n+m-n}}$ for every $m\geq n$. 
\end{remark}

\section{The (NC) and (LIE) conditions}\label{sec:NC} 

The goal of the next two sections is to show that the assumptions of
\Cref{thm:Piras} hold in the case relevant to the proof of
\Cref{thm:introbasecase}. Property (NC) corresponds to
\Cref{thm:NC} \textup{d)} below, while Property (LIE) is dealt with in
\Cref{prop:lie}.
\subsection{The normal closure property}
Consider a tower of algebraic extensions $\QQ\subseteq \KK\subseteq \FF\subseteq \LL$, with $\LL/\KK$ and $\FF/\KK$ Galois.
Let $v$ be a finite place of $\KK$. 
Since $\Gal(\LL/\FF)$ is a normal subgroup of $\Gal(\LL/\KK)$, it is easy to see that any two decomposition subgroups of $\Gal(\LL/\FF)$ at any place of $\LL$ extending $v$ are conjugate in $\Gal(\LL/\KK)$.
Their normal closures in $\Gal(\LL/\KK)$ are therefore equal. 
Thus, in order to study such normal closures, we may focus our attention on a fixed place of $\LL$ over $v$, or, equivalently, on a fixed field embedding $\overline{\QQ}\hookrightarrow \overline{\QQ_p}$ whose restriction to $\KK$ is associated to the place $v$. 
\begin{theorem} \label{thm:NC}
Choose a rational prime $p$, an algebraic extension $\KK$ of $\QQ$ containing $\mug_{2p}$, and a finitely generated subgroup $\Gamma\subseteq \KK^\times$. 
Let us fix a field embedding $\overline{\QQ} \hookrightarrow \overline{\QQ_p}$, and denote by $K$ the topological closure of $\KK$ in $\overline{\QQ_p}$ as well as $\KK_t=\KK(\Gamma^{1/p^t})$ and $K_t=K(\Gamma^{1/p^t})$ for all integers $t\geq 0$.
If $K$ does not contain $\mug_{p^\infty}$, then there are an integer $n\geq 2$ and a subgroup $\tilde{\Gamma}\subseteq \Gamma^{1/p^n}\cdot \mug_{p^l}$, where $\mug_{p^l}:=K_n\cap \mug_{p^\infty}$, such that, letting $ \EE:=\KK_n(\zeta_{p^l})$, we have:
\begin{enumerate} [label=\alph*)]
\item $\KK(\Gamma^{p-\mathrm{div}})= \EE( \tilde{\Gamma}^{p-\mathrm{div}})$; 
\item the topological closure of $\EE(\tilde{\Gamma}^{1/p^m})$ in $\overline{\QQ_p}$ is $K_n(\tilde{\Gamma}^{1/p^m})$ for all integers $m\geq 0$;
\item $K_n(\tilde{\Gamma}^{1/p^m})\cap \mug_{p^\infty}=\mug_{p^m}$ for all integers $m\geq l$; 
\item for all $m\geq l$, the group $\Gal(\EE(\tilde{\Gamma}^{1/p^{m+1}})/\EE(\tilde{\Gamma}^{1/p^m}))$ equals the normal closure of the decomposition subgroup $\Gal(K_n(\tilde{\Gamma}^{1/p^{m+1}})/K_n(\tilde{\Gamma}^{1/p^m}))$ in $\Gal(\EE(\tilde{\Gamma}^{1/p^{m+1}})/\EE)$.
\end{enumerate}
\end{theorem}

\begin{proof}
By applying Proposition \ref{prop:degradext} to the fields $\KK$ and $K$, we get integers $n\geq 2$ and  $s, s'\geq 1$ such that $[\KK_m:\KK_n]=p^{s(m-n)}$ and $[K_m:K_n]=p^{s'(m-n)}$ for all integers $m\geq n$.  Moreover,  Proposition \ref{prop:degradext} provides elements $\gamma_2,\dots,\gamma_s,\delta_2,\dots,\delta_{s'}\in \Gamma^{1/p^n}$
such that $\KK_m=\KK_n(\zeta_{p^k}^{1/p^{m-n}}, \gamma_2^{1/p^{m-n}},\dots, \gamma_s^{1/p^{m-n}})$ and $K_m=K_n(\zeta_{p^l}^{1/p^{m-n}}, \delta_2^{1/p^{m-n}},\dots, \delta_{s'}^{1/p^{m-n}})$ for all integers $m\geq n$. 
Here, $\mug_{p^k}:=\KK_n\cap \mug_{p^\infty}$ and note that $k\leq l$ since $\zeta_{p^k}\in K_n$. 

We now construct $\tilde{\Gamma}$. 
Let $i\in\{2,\dots,s\}$. 
Since $\KK_{n+l}$ is a subfield of $K_{n+l}$, we infer that $\gamma_i$ is a $p^l$-th power in $K_{n+l}$.
Remark \ref{rem:kummer} shows that $\gamma_i\zeta_{p^l}^{u_i} \in K_n^{p^l} \cdot \langle \delta_2,\dots, \delta_{s'}\rangle$
 for some integer $u_i\in\{0,\dots,p^l-1\}$. 
Write $\tilde{\Gamma}$ for the group generated by $\gamma_2 \zeta_{p^l}^{u_2}, \dots, \gamma_s\zeta_{p^l}^{u_s}$.

Conclusion $a)$ becomes clear since \[ \EE(\tilde{\Gamma}^{p-\mathrm{div}})=\bigcup_{m\geq n} \KK_n(\zeta_{p^l}^{1/p^{m-n}}, \tilde{\Gamma}^{1/p^{m-n}})=\bigcup_{m\geq n} \KK_m(\zeta_{p^l}^{1/p^{m-n}})=\KK(\Gamma^{p-\mathrm{div}}). \]

The topological closure of $\EE$ in $\overline{\QQ_p}$ being $K_n(\zeta_{p^l})=K_n$, we infer that that of $\EE(\tilde{\Gamma}^{1/p^m})$ is $K_n(\tilde{\Gamma}^{1/p^m})$ for all integers $m\geq 0$. This shows $b)$.

We now show $c)$. 
Since 
\[[K_{n+l}:K_n]=[K_n(\zeta_{p^l}^{1/p^l}, \delta_2^{1/p^l},\dots, \delta_{s'}^{1/p^l}) : K_n]=p^{s'l},\] we must have $\zeta_{p^{l+1}}\not\in K_n(\delta_2^{1/p^l},\ldots \delta_{s'}^{1/p^l})$.
The construction of $\tilde{\Gamma}$ is made in such a way that $K_n(\tilde{\Gamma}^{1/p^l})\subseteq K_n(\delta_2^{1/p^l},\dots, \delta_{s'}^{1/p^l})$; whence $K_n(\tilde{\Gamma}^{1/p^l})\cap\mug_{p^\infty}=\mug_{p^l}$. 
We get $K_n(\tilde{\Gamma}^{1/p^{l+j}})\cap\mug_{p^\infty}=\mug_{p^{l+j}}$ for all integers $j\geq 0$ since $\zeta_{p^{l+j}}\in \tilde{\Gamma}^{1/p^{l+j}}$ and since the Galois group of the extension $K_n(\tilde{\Gamma}^{1/p^{l+j}})/K_n(\tilde{\Gamma}^{1/p^l})$ has exponent $p^j$. 

Let us make an observation before proving $d)$.
Since the degree of the extension $\KK_{n+1}/\KK_n$ is $p^s$, then  $\KK_n(\gamma_2^{1/p},\dots,\gamma_s^{1/p})/\KK_n$ is an extension of degree $p^{s-1}$, therefore $\zeta_{p^{k+1}}\notin \KK_n(\gamma_2^{1/p},\dots,\gamma_s^{1/p})$. 
The latter implies that $\KK_n(\zeta_{p^{l+1}})$ and $\KK_n(\gamma_2^{1/p},\dots,\gamma_s^{1/p})$ are linearly disjoint over $\KK_n$.
We conclude that the extension $\KK_n(\zeta_{p^{l+1}}, \tilde{\Gamma}^{1/p})/\KK_n(\zeta_{p^{l+1}})$ has degree $p^{s-1}$, and therefore 
\begin{equation} \label{eq:stabdegreeforE}
[\EE(\tilde{\Gamma}^{1/p}):\EE]=p^{s-1}.
\end{equation}
We are now ready to establish $d)$. We pick an integer $m\geq l$.
Write $\MM$ for the subfield of $\EE(\tilde{\Gamma}^{1/p^{m+1}})$ fixed by the normal closure of the decomposition subgroup $\Gal(K_n(\tilde{\Gamma}^{1/p^{m+1}})/K_n(\tilde{\Gamma}^{1/p^m}))$ in $\Gal(\EE(\tilde{\Gamma}^{1/p^{m+1}})/\EE)$. 
Clearly, $\MM$ is Galois over $\EE$ and it contains $\EE(\tilde{\Gamma}^{1/p^m})$ since the latter is Galois over $\EE$. 
It is also a subfield of $K_n(\tilde{\Gamma}^{1/p^m})$ by construction; in particular, $\MM\cap\mug_{p^\infty}\subseteq\mug_{p^m}$.
Thanks to the observations made above, Lemma \ref{lmm:Mgaloisifroots} with $K=\EE, \Gamma=\tilde{\Gamma}$, and $M=\MM$ yields $\MM=\EE(\tilde{\Gamma}^{1/p^m})$. 

The theorem follows. 
\end{proof} 

\begin{remark}\label{rem:Eroots}
It is an obvious consequence of Theorem \ref{thm:NC} $b)$ and $c)$ that one has $\EE(\tilde\Gamma^{1/p^m})\cap \mug_{p^\infty} =\mug_{p^m}$ for every $m\ge l$.
\end{remark}

\subsection{The Lie filtration on $\Gal(\EE(\tilde{\Gamma}^{p-\mathrm{div}})/\EE$)} 
We show that the Galois group of $\EE(\tilde{\Gamma}^{p-\mathrm{div}})/\EE$ has a structure of $p$-adic Lie group, with a Lie filtration provided by Theorem \ref{thm:NC}. This will be needed in order to apply Theorem \ref{thm:Piras}.

Recall that the \textit{lower $p$-series} $(H_i)_{i\ge 0}$ of a topological group $H$ is defined recursively by setting $H_{i+1}$ to be the closure of $H_i^p[H,H_i^p]$ in $H$, the brackets denoting commutators.

By \cite[Theorem 8.32]{AnalyticpropBook}, a topological group $G$ is a $p$-adic Lie group (a ``$p$-adic analytic group'' in the terminology of \textit{loc.cit.}) if and only if it admits a uniform pro-$p$ open subgroup, in the sense of \cite[Definition 4.1]{AnalyticpropBook}. In this case, a \textit{Lie filtration} of $G$ is a filtration $(G_i)_{i\ge 0}$ of $G$ by open normal subgroups, for which there exist $i_0\ge 1$ such that $G_{i_0}$ is a uniform pro-$p$ open subgroup of $G$ and $(G_i)_{i\ge i_0}$ is the lower $p$-series of $G_{i_0}$. 

Let $\EE,\wtl\Gamma$ be as in Theorem \ref{thm:NC}. 
Let $G=\Gal(\EE(\tilde\Gamma^{p-\mathrm{div}})/\EE)$.
For every $m\geq 0$, define $G_m=\Gal(\EE(\tilde\Gamma^{p-\mathrm{div}})/\EE(\tilde\Gamma^{1/p^m}))$, so that $G_0=G$.

\begin{proposition}\label{prop:lie}
The group $G$ is a $p$-adic Lie group of positive dimension, with Lie filtration $(G_m)_{m\ge 0}$. 
\end{proposition}

\begin{proof}
With the notation from the proof of Theorem \ref{thm:NC}, set $\tilde\gamma_i=\gamma_i\zeta_{p^l}^{u_i}$ for $2\le i\le s$, so that $\tilde\Gamma\subseteq\EE^\times$ is generated by $\tilde\gamma_2,\ldots,\tilde\gamma_s$. Note that, if $s=1$, then $\tilde\Gamma$ is trivial and there are no $\tilde\gamma_i$, but the proof goes through; simply, the factor $\ZZ_p^{s-1}$ appearing below is trivial. 
Let $(\zeta_{p^m})_{m\ge 1}$ and $(\tilde\gamma_i^{1/p^{m}})_{m\ge 1}$, for $2\le i\le s$, be compatible systems of $p$-power roots. Then $G$ acts on $(\zeta_{p^{m}})_{m\ge 1}$ via the restriction $\chi$ of the $p$-adic cyclotomic character to $G$. 
For $2\le i\le s$, any $g\in G$ acts on $(\tilde\gamma_i^{1/p^{m}})_{m\ge 1}$ as multiplication with a compatible system of $p$-power roots of unity, i.e. there exists $\alpha_i(g)\in\Z_p$ such that $g(\tilde\gamma_i^{1/p^{m}})=\zeta_{p^{m}}^{\alpha_i(g)}\tilde\gamma_i^{1/p^{m}}$ for every $i,m$. Each map $\alpha_i:G\to\ZZ_p$ is surjective since $\tilde\gamma_i\in \EE$ and $\tilde\gamma_i^{1/p}\notin\EE$, the latter being a direct consequence of \eqref{eq:stabdegreeforE}.

Define a semidirect product $\ZZ_p^\times\ltimes\ZZ_p^{s-1}$ where $\ZZ_p^\times$ acts on $\ZZ_p^{s-1}$ via scalar multiplication. Then the map
\[ \varphi:=(\chi,\alpha_2,\ldots,\alpha_s):G\to\ZZ_p^\times\ltimes\ZZ_p^{s-1} \]
is an injective homomorphism of groups: compatibility with the group operations is checked directly, and injectivity from the fact that the action of $g\in G$ on $\EE(\tilde\Gamma^{p-\mathrm{div}})$ is uniquely determined by its action on the compatible systems $(\zeta_{p^m})_{m\ge 1}$ and $(\tilde\gamma_i^{1/p^m})_{m\ge 1}$.

Embed $\ZZ_p^\times\ltimes\ZZ_p^{s-1}$ as a closed subgroup $H\subseteq\GL_s(\ZZ_p)$ by mapping the $s$ coefficients in the semidirect product to the $s$ entries of the first matrix row and letting the other coefficients be $1$ on the principal diagonal and $0$ outside; any ordering of the element in the first row will do as long as the copy of $\ZZ_p^\times$ is mapped to the first entry. The group $\GL_s(\ZZ_p)$ is a $p$-adic Lie group, with Lie filtration given by the principal congruence subgroups $\Gamma(p^m):=\ker(\GL_s(\ZZ_p)\to\GL_s(\ZZ_p/p^m\ZZ_p))$, for $m\ge 1$. 
Set $H_m:=\Gamma(p^m)\cap H$; clearly $H_m$ is the image of $(1+p^m\ZZ_p)\ltimes(p^m\ZZ_p)^{s-1}$ under our embedding.

Thanks to Remark \ref{rem:Eroots}, $\chi\vert_{G_m}:G_m\to\ZZ_p^\times$ induces a surjection $G_m\onto 1+p^m\ZZ_p$ for $m\ge l$, so that $\varphi\vert_{G_m}$ induces an isomorphism $G_m\cong (1+p^m\ZZ_p)\ltimes(p^m\ZZ_p)^{s-1}$, and $\varphi(G_m)=H_m$, for $m\ge l$. 

Since $\varphi(G)$ is a closed subgroup of $\GL_s(\ZZ_p)$, by \cite[Section 4, Exercise 14]{AnalyticpropBook} (or a direct calculation), there exists $m_0\ge l$ such that $H_{m_0}\cap\varphi(G)$ is a uniform pro-$p$ subgroup of $\varphi(G)$, so that $(H_m\cap\varphi(G))_{m\ge m_0}$ is the lower $p$-series of $H_{m_0}$, and $(H_m\cap\varphi(G))_{m\ge 0}$ is a Lie filtration of $\varphi(G)$. 
Since $H_m\cap\varphi(G)=H_m =\varphi(G_m)$ for $m\ge l$, we deduce that $(G_m)_{m\geq 0}$ is a Lie filtration of $G$.

Finally, note that a $p$-adic Lie group is of positive dimension if and only if it is infinite, and since $\mug_{p^\infty}\subseteq\EE(\tilde\Gamma^{p-\mathrm{div}})$ but $\mug_{p^\infty}\not\subseteq\EE$, the group $G$ is infinite.
\end{proof}

\section{The (LTR) condition}\label{sec:LPTR}
A valued field \(K\) is said to be Henselian if its valuation extends uniquely to every finite separable extension of \(K\). Note that any algebraic extension of a perfect Henselian field is Henselian too. 

For any Galois extension $L/K$ of Henselian fields, we denote by $I(L/K)$ its inertia group. By definition of the Krull topology, a subgroup of $\Gal(L/K)$ is open in $\Gal(L/K)$ if its fixed field in $L$ is a finite extension of $K$. 
In particular, $I(L/K)$ is open in $\Gal(L/K)$ if and only if the inertia degree of the extension $L/K$ is finite. 

 The aim of this section is to show the following:

\begin{theorem} \label{prop G/I}
Let $K$ be a Henselian field of characteristic $0$ containing $\mug_{2p}$, but not $\mug_{p^\infty}$, and whose residue field lies in $\overline{\FF_p}$. 
Let $\Gamma\subseteq K^\times$ be a finitely generated subgroup.
Then the inertia subgroup $I(K(\Gamma^{p-\mathrm{div}})/K)$ is open in $\Gal(K(\Gamma^{p-\mathrm{div}})/K)$ if and only if $I(K(\mug_{p^\infty})/K)$ is open in $\Gal(K(\mug_{p^\infty})/K)$. 
\end{theorem}

Before proving \Cref{prop G/I}, we state the consequence we are interested in. 

\begin{cor} \label{cor LPTR}
Let $p$ be a rational prime, $F/\QQ_p$ an algebraic extension, and $\Gamma\subseteq F^\times$ a finitely generated subgroup. If $F$ is a discrete valuation field, then there is an integer $n_0$ such that the extension $F(\Gamma^{p-\mathrm{div}})/F(\Gamma^{1/p^{n_0}})$ is totally ramified. 
\end{cor}

\begin{proof}
Write $e$ for the ramification index of $F$ over $\Q_p$, and let $n\geq 1$ be an integer. 
Since the extension $\Q_p(\mug_{p^n})/\Q_p$ is totally ramified of degree $(p-1)p^{n-1}$, we infer that $F(\mug_{p^n})/F$ is an extension of degree at most $(p-1)p^{n-1}$ whose ramification index is at least $(p-1)p^{n-1}/e$. 
This forces its inertia degree to be at most $e$, and this holds for each $n$. It follows that $F(\mug_{p^\infty})/F(\mug_{2p})$ has finite inertia degree. Notice that $F(\mug_{2p})$ is a discrete valuation field, so that $\mug_{p^\infty}\not\subseteq F(\mug_{2p})$. Moreover its residue field lies in $\overline{\FF_p}$, so that Theorem \ref{prop G/I}, applied to $K=F(\mug_{2p})$,  
ensures us that the inertia degree of the extension $F(\Gamma^{p-\mathrm{div}})/F$ is finite.
Therefore, $F(\Gamma^{1/p^n})/F(\Gamma^{1/p^{n-1}})$ has inertia degree $1$ for all but finitely many $n$. Setting $n_0$ to be the maximum of the integers $n$ for which this extension does not have inertia degree $1$, we get that $F(\Gamma^{p-\mathrm{div}})/F(\Gamma^{1/p^{n_0}})$ is totally ramified.
\end{proof}

The forward direction of Theorem \ref{prop G/I} is obvious, so we only focus on the converse in the sequel of this section. 
Fix for this section a Henselian field $K$ of characteristic $0$ containing $\mug_{2p}$, but not $\mug_{p^\infty}$, and whose residue field is an algebraic extension of $\FF_p$. 
We also fix a finitely generated subgroup $\Gamma\subseteq K^\times$ and we suppose that $I(K(\mug_{p^\infty})/K)$ is open in $\Gal(K(\mug_{p^\infty})/K)$. 
Of course, $I(L(\mug_{p^\infty})/L)$ is open in $\Gal(L(\mug_{p^\infty})/L)$ for all finite extensions $L/K$. 
Note that the theorem would follow if we are able to prove that $I(L(\Gamma^{p-\mathrm{div}})/L)$ is open in $\Gal(L(\Gamma^{p-\mathrm{div}})/L)$ for at least one finite extension $L/K$. 
The construction of a suitable field $L$ is the purpose of the next lemma.

\begin{lemma} \label{lem:form.Ss-sec4.1}
There exist a finite extension $L/K$ and $s\geq 1$ elements $\zeta_{p^l},\gamma_2,\ldots,\gamma_s\in L\cap\Gamma^{p-\mathrm{div}}$ such that, letting $L_m=L(\zeta_{p^l}^{1/p^{m}}\gamma_2^{1/p^{m}},\dots,\gamma_s^{1/p^{m}})$ for each $m\ge0$,  we get: 
\begin{enumerate} [label=\alph*)]
\item $L(\mug_{p^\infty})/L$ is totally ramified;
\item $L(\Gamma^{p-\mathrm{div}})=L(\mug_{p^\infty}, \gamma_2^{1/p^\infty},\dots,\gamma_s^{1/p^\infty})$;
\item  $\Gal(L_{m+1}/ L_m)\cong(\Z/p\Z)^s$ for each $m\ge0$. 
\end{enumerate}
\end{lemma}
\begin{proof}
Set $K_m=K(\Gamma^{1/p^m})$ for every integer $m\geq 0$. 
Proposition \ref{prop:degradext}  applied to the extension $K(\Gamma^{p-\mathrm{div}})/K$ provide integers $n\geq 2$ and $s\geq 1$, as well as a subgroup $\tilde\Gamma=\langle\zeta_{p^{k_n}},\gamma_2,\dots,\gamma_s\rangle$ of $K_n^\times$, where $\mug_{p^{k_n}}=K_n\cap\mug_{p^\infty}$ and $\gamma_2,\dots,\gamma_s\in\Gamma^{1/p^n}$, such that $K_m=K_n(\tilde\Gamma^{1/p^{m-n}})$ and $K_{m+1}/K_m$ has degree $p^{s}$ for all integers $m\geq n$.

The fact that the inertia group $I(K_n(\mug_{p^\infty})/K_n)$ is open in $\Gal(K_n(\mug_{p^\infty})/K_n)$ by assumption means that there is a smallest integer greater than or equal to $k_n$, say $l$, such that $K_n(\mug_{p^\infty})/K_n(\zeta_{p^l})$ is totally ramified. Let now $L=K_n(\zeta_{p^l})$. 

Clearly $L(\mug_{p^\infty})=K_n(\mug_{p^\infty})$; this proves $(a)$.

On the other hand, our construction ensures that $L(\Gamma^{p-\text{div}})=K(\Gamma^{p-\text{div}})=K_n(\tilde{\Gamma}^{p-\text{div}})=L(\mug_{p^\infty}, \gamma_2^{1/p^\infty},\dots,\gamma_s^{1/p^\infty})$;  this proves $(b)$.

We now show $c)$. 
  Let $j\geq 0$ be an integer. 
  Thanks to Remark \ref{rem:km}, we know that $K_{j+n}\cap\mug_{p^\infty}=\mug_{p^{j+k_n}}$. 
  Since $K_{j+n}(\zeta_{p^{j+l}})=L_j$, we infer that $L_j/K_{j+n}$ is an extension of degree $p^{l-k_n}$. 
 Thus, given an integer $m\geq 0$, one has \[[L_{m+1}:L_m]=\frac{[L_{m+1}:K_{m+1+n}]}{[L_m:K_{m+n}]}[K_{m+n+1}:K_{m+n}]=p^s\] and $c)$ follows by noticing that $L_{m+1}/L_m$ is a Kummer extension of exponent $p$.
\end{proof}

Note that $\zeta_{p^{m+1+l}}\notin L_m$ for all integers $m\geq 0$ by Lemma \ref{lem:form.Ss-sec4.1} $c)$.

 The next lemma shows that $I(L(\Gamma^{p-\mathrm{div}})/L)$ is open in $\Gal(L(\Gamma^{p-\mathrm{div}})/L)$ if one extension in the tower $L=L_0\subseteq L_1 \subseteq L_2 \subseteq \cdots \subseteq  L(\Gamma^{p-\mathrm{div}})$ is totally ramified. 
  
\begin{lemma} \label{lmm:ramimpliestotram}
 Assume that there is an integer $n\geq 0$ for which $L_{n+1}/L_n$ is totally ramified. Then $L(\Gamma^{p-\mathrm{div}})/L_n$ is totally ramified.
 \end{lemma}
 
 \begin{proof}
 It suffices to show that $L_m/L_n$ is totally ramified for all $m>n$. Assume by contradiction that it is not the case and write $M$ for the maximal unramified extension of $L_n$ contained in $L_m$. 
 The degree of the extension $L_m/L_n$ being a power of $p$, the same holds for the Galois extension $M/L_n$. Its Galois group then has a subgroup of order $d$ for every divisor $d$ of $[M :L_n]$. The Galois correspondence now asserts that $M/L_n$ contains an intermediate subfield $E$ of degree $p$ over $L_n$. The extension $L_{n+1}/L_n$ is totally ramified and $E/L_n$ is unramified, therefore these two extensions are linearly disjoint, and $\Gal(L_{n+1}E/L_n)\cong\Gal(L_{n+1}/L_n)\times\Gal(E/L_n)\cong(\ZZ/p\ZZ)^{s+1}$. This gives a contradiction since this group is a quotient of $\Gal(L(\Gamma^{p-\text{div}})/L_n)\cong (1+p^{n+l}\ZZ_p) \ltimes \ZZ_p^{s-1}$, which can be generated by only $s$ elements.  
\end{proof}
 
 In other words, Theorem \ref{prop G/I} would follow if the extension $L_{n+1}/L_n$ was totally ramified for some $n$. From now on, we assume that the extension $L_{n+1}/L_n$ is not totally ramified for any $n$, and we derive a contradiction.  
 
 The next lemma states a general fact about unramified extensions. 
 
 \begin{lemma} \label{lmm:unrextcyclic}
 Let $M$ be a Henselian field of characteristic $0$ whose residue field lies in $\overline{\FF_p}$. 
 Then every finite unramified extension of $M$ is cyclic. 
 \end{lemma} 
 
 \begin{proof}
 Let $E/M$ be a finite unramified extension. 
 By the general theory of unramified extensions, its Galois group is isomorphic to $\Gal(\kappa_E/\kappa_M)$, where $\kappa_E$ denotes the residue field of $E$ and $\kappa_M$ that of $M$ \cite[Chapter II, Proposition 9.9]{Neukirch1999}. 
 The extension $\kappa_E/\kappa_M$ is therefore finite. 
 Since $\Gal(\kappa_E/\FF_p)$ is pro-cyclic, and $\Gal(\kappa_E/\kappa_M)$ is a finite (hence closed) subgroup of a pro-cyclic group, the latter is therefore cyclic.
  \end{proof}

Define $M_m$ as the maximal unramified extension of $L$ contained in $L_m$.  
The inclusion $L_m \subseteq L_{m+1}$ leads to $M_m\subseteq M_{m+1}$ for all $m\geq 0$. 
 We now compute its Galois group over $L$ thanks to Lemma \ref{lmm:unrextcyclic}. 
  Recall that we assumed by contradiction that the extension $L_{n+1}/L_n$ is not totally ramified for any $n$.
  
 \begin{lemma} \label{lmm iso}
 We have $\Gal(M_m/L) \cong \Z/p^m\Z$ for all $m\geq 0$. 
 \end{lemma}
 
 \begin{proof}
 Because $\Gal(M_m/L)$ is cyclic by Lemma \ref{lmm:unrextcyclic}, it is enough to show that $M_m/L$ is a field extension of degree $p^m$. 
By the multiplicativity formula for inertia degrees, it suffices to show that the inertia degree of $L_{n+1}/L_n$ is $p$ for all $n\in\{0,\dots, m-1\}$. 
 
 Let $E$ be the maximal unramified extension of $L_n$ contained in $L_{n+1}$, which is different from $L_n$ by assumption. On the one hand, $\Gal(L_{n+1}/L_n)$ is an abelian group of exponent $p$. On the other hand, $\Gal(E/L_n)$ is a cyclic group by Lemma \ref{lmm:unrextcyclic}. It is therefore isomorphic to $\Z/p\Z$ as a cyclic group of exponent $p$. 
 \end{proof}

We are now ready to obtain the desired contradiction, and thus prove Theorem \ref{prop G/I}.
First, given any integer $m\geq 0$, the extension $M_m/L$ is unramified while $L(\mug_{p^\infty})/L$ is totally ramified. 
They are therefore linearly disjoint, and so $M_m\cap L(\mug_{p^\infty})=L$. 
In particular, $\zeta_{p^{l+1}}\notin M_m$ since $L\cap\mug_{p^\infty}=\mug_{p^l}$.

We now pick any integer $m\geq l+1$, and let us obtain the contradiction $\zeta_{p^m}\in M_m$.
Lemma \ref{lmm iso} asserts that $M_m/M_{m-1}$ is a field extension of degree $p$. 
Since $L$ contains the $p$-th roots of unity, one has $M_m=M_{m-1}(\varepsilon^{1/p})$ for some $\varepsilon\in M_{m-1}$. 
The extension $L_m/M_{m-1}(\zeta_{p^{m+l}})$ being Kummer, Remark  \ref{rem:kummer} yields \[ \varepsilon^{1/p} = b \prod_{i=2}^s \gamma_i^{u_i/p^m}\] for some $b\in M_{m-1}(\zeta_{p^{m+l}})$ and some integers $u_2,\dots, u_s$.
Note that $b^{p^m}\in M_{m-1}$. 

Let $j$ be the least non-negative integer for which $b^{p^j}\in M_{m-1}$. 
Of course, $j\leq m$.
The minimal polynomial of $b$ over $M_{m-1}$ is therefore $X^{p^j}-b^{p^j}$.

The fact that $M_{m-1}(\zeta_{p^{m+l}})/M_{m-1}$ is a cyclic extension with $\zeta_{p^{l+1}}\notin M_{m-1}$ implies, for degree reasons, that $M_{m-1}(b)=M_{m-1}(\zeta_{p^{l+j}})$. 
From \cite[Lemma 2.3]{Velez1980}, we get $b=c\zeta_{p^{l+j}}^a$ for some $c\in M_{m-1}$ and some integer $a$.
Thus, $M_m=M_{m-1}(\varepsilon^{1/p})=M_{m-1}(\delta_m)$, where \[\delta_m= \zeta_{p^{l+j}}^a \prod_{i=2}^s \gamma_i^{u_i/p^m}.\] 

Note that $\delta_m^{p^m}\in L$ since $j\leq m$.
Because $L=M_0\subseteq M_1\subseteq \dots \subseteq M_m$, it follows from Lemma \ref{lmm iso} that any intermediate subfield of the extension $M_m/L$ is of the form $M_D$ for some $D\in\{0,\dots,m\}$. The fact that $\delta_m \in M_m \setminus M_{m-1}$ leads to  $M_m=L(\delta_m)$. 
Since $M_m/L$ is a Galois extension of degree $p^m$, it follows that the Galois orbit of $\delta_m$ over $L$, namely $\mug_{p^m} \cdot \delta_m$, lies in $M_m$; whence $\zeta_{p^m}\in M_m$, a contradiction. \qed 

\section{Proof of the main tool}\label{sec:mainproof}

We are now ready to prove the following result. 
\begin{theorem} \label{thm:basecase}
Let $\FF$ be an algebraic extension of $\QQ$,  $\Gamma\subseteq \FF^\times$ a finitely generated subgroup, and $p$ a rational prime. 
Assume that:
\begin{itemize}
\item[a)] $\FF$ is Galois over some number field $\KK$; 
\item[b)] $\FF'\setminus\Gamma^{\mathrm{div}}$ has the property \B for every finite extension $\FF'/\FF$;
\item[c)] there is a field embedding $\iota: \overline{\Q}\hookrightarrow \overline{\Q_p}$ such that the topological closure of $\iota(\FF)$ in $\overline{\Q_p}$ is a discrete valuation field.  
\end{itemize}
Then $\FF(\Gamma^{p-\mathrm{div}})\setminus\Gamma^{\mathrm{div}}$ has the property \B.
\end{theorem}

\begin{proof}

To ease notation, we fix the field embedding $\iota$, which allows us to see $\overline{\QQ}$ as a subfield of $\overline{\QQ_p}$ according to our convention stated in the notation subsection. 

Let $F$ be the topological closure of $\FF$ in $\overline{\QQ_p}$. 
Since $F$ is a discrete valuation field by assumption, Corollary \ref{cor LPTR} proves that the extension $F(\Gamma^{p-\mathrm{div}})/F(\Gamma^{1/p^{n_0}})$ is totally ramified for some integer $n_0\geq 2$.  Note that the topological closure of any finite extension of $\FF$ in $\overline{\QQ_p}$ is a discrete valuation field as it is a finite extension of $F$. In particular, it does not contain $\mug_{p^\infty}$. 

We now use Theorem \ref{thm:NC} with the prime $p$, the field $\FF(\Gamma^{1/p^{n_0}})$, the group $\Gamma^{1/p^{n_0}}$, and the field embedding $\iota$. 
Hence, there exist integers $l\geq n \geq 2$ as well as a subgroup $\tilde{\Gamma}\subseteq \Gamma^{1/p^{n_0}}\cdot \mug_{p^l}$ such that $\FF(\Gamma^{p-\mathrm{div}})= \EE(\tilde{\Gamma}^{p-\mathrm{div}})$, where $\EE=\FF(\Gamma^{1/p^{n+n_0}}, \zeta_{p^l})$. 
Setting $\LL_m=\EE(\tilde{\Gamma}^{1/p^{m+l}})$ and $L_m=F(\Gamma^{1/p^{n+n_0}})(\tilde{\Gamma}^{1/p^{m+l}})$ for all integers $m\geq 0$, Theorem \ref{thm:NC} also claims that  for all integers $m\geq 0$, the topological closure of $\LL_m$ in $\overline{\QQ_p}$ equals $L_m$ and the normal closure of $\Gal(L_{m+1}/L_m)$ in $\Gal(\LL_{m+1}/\EE)$ is $\Gal(\LL_{m+1}/\LL_m)$. 

This maneuver is to ensure that the tower of algebraic extensions \[ \KK(\Gamma, \tilde{\Gamma})\subseteq \EE \subseteq \LL_0\subseteq \LL_1 \subseteq \LL_2 \subseteq \dots \subseteq  \EE(\tilde{\Gamma}^{p-\mathrm{div}})=\FF(\Gamma^{p-\mathrm{div}})\]  satisfies the properties (GAL,  DVF, LTR, NC) introduced in Theorem \ref{thm:Piras}.  The property (LIE) is also satisfied by Proposition \ref{prop:lie}. 

In conclusion, Theorem \ref{thm:Piras} implies that the set \[ \mathcal{Z}=  \bigcup_{m\geq 1} \{\alpha\in\LL_m \ |\ \exists\sigma\in \Gal(\FF(\Gamma^{p-\mathrm{div}})/\LL_{m-1}) \hbox{ such that } \sigma(\alpha)/\alpha\not\in\mug_{p^\infty}\} \] has the property \B. 

Assume by contradiction that there is a sequence of terms $\alpha_i\in \FF(\Gamma^{p-\mathrm{div}})^\times  \setminus \Gamma^{\mathrm{div}}$ such that $h(\alpha_i)\to 0$. 
For all $i$, denote by $m_i$ the smallest non-negative integer such that $\alpha_i\in \LL_{m_i}^\times\cdot \tilde{\Gamma}^{p-\mathrm{div}}$. 
This integer exists since $\alpha_i\in \LL_m^\times \subseteq \LL_m^\times \cdot \tilde{\Gamma}^{p-\mathrm{div}}$ for some $m$.

First, assume that the sequence $(m_i)_i$ is bounded, say by $M$, that is, every $\alpha_i$ lies in $\LL_M^\times \cdot \tilde{\Gamma}^{p-\mathrm{div}}$.
Since $\LL_M/\FF$ is a finite extension, condition $b)$ claims that the set $\LL_M \setminus \Gamma^{\mathrm{div}}$ has the property \B.
By Theorem \ref{lem:Remstronggen}, the set $(\LL_M^\times \cdot \Gamma^{\mathrm{div}}) \setminus \Gamma^{\mathrm{div}}$ has the property \B too, contradicting the fact that $h(\alpha_i)\to 0$ since $\tilde{\Gamma}^{p-\mathrm{div}}\subseteq \Gamma^{\mathrm{div}}$. 

The sequence $(m_i)_i$ is therefore unbounded and by taking a suitable subsequence if needed, we can assume without loss of generality that $m_i\to +\infty$.
By construction, we have $\alpha_i=\beta_i \gamma_i$ for some $\beta_i\in \LL_{m_i}$ and some $\gamma_i\in \tilde{\Gamma}^{p-\mathrm{div}}$. 
Fix a generating set $\{b_1,\dots,b_r\}$  of $\tilde{\Gamma}$. 
We can thus express every $\gamma_i$ as 
 \[ \gamma_i= \zeta_{p^{n_i}}^{k_{0,i}} \prod_{j=1}^r b_j^{k_{j,i}/p^{n_i}}\] for some integers $k_{j,i}$ and some integer $n_i>m_i$.
Let $k_{j,i}=q_{j,i} p^{n_i-m_i}+r_{j,i}$ denote the Euclidean division of $k_{j,i}$ by $p^{n_i-m_i}$. 
Note that $r_{j,i}/p^{n_i} < 1/p^{m_i}$.
A direct computation shows that \[ \delta_i:=\alpha_i \zeta_{p^{n_i}}^{-r_{0,i}} \prod_{j=1}^r b_j^{-r_{j,i}/p^{n_i}} = \beta_i \zeta_{p^{m_i}}^{q_{0,i}} \prod_{j=1}^r b_j^{q_{j,i}/p^{m_i}} \in \LL_{m_i}^\times \setminus \Gamma^{\mathrm{div}} \] since $\alpha_i\notin \{0\} \cup \Gamma^{\mathrm{div}}$. 
By the basic properties of the Weil height, we have \[ h(\delta_i) \leq  h(\alpha_i)+ \sum_{j=1}^r \frac{r_{j,i}}{p^{n_i}} h(b_j) < h(\alpha_i)+ \frac{1}{p^{m_i}} \sum_{j=1}^r h(b_j). \]
 It follows that $h(\delta_i)\to 0$, since \(m_i\to+\infty\). 
Therefore, $\delta_i\notin \mathcal{Z}$ for all $i$ large enough, and since $\delta_i\in \LL_{m_i}$, we get $\sigma(\delta_i)/\delta_i \in \mug_{p^\infty}$ for all $\sigma\in \Gal(\LL_{m_i}/\LL_{m_i-1}) $.  
Using Lemma \ref{lem:orbitagen2} with the field $\LL_{m_i-1}$ and the group $\tilde{\Gamma}^{1/p^{m_i+l-1}}$ gives $\delta_i\in \tilde{\Gamma}^{p-\mathrm{div}}\cdot \LL_{m_i-1}$. 
We conclude $\alpha_i \in \delta_i\cdot \tilde{\Gamma}^{p-\mathrm{div}}\subseteq \LL_{m_i-1}\cdot \tilde{\Gamma}^{p-\mathrm{div}}$, contradicting the minimality of $m_i$.
\end{proof}

\bigskip

\bibliographystyle{abbrv}
\bibliography{BOG}

\bigskip

\begin{small}
\noindent\textsc{Andrea Conti} -- IWR, Heidelberg University, Im Neuenheimer Feld 205, 69120 Heidelberg, Germany, \url{andrea.conti@iwr.uni-heidelberg.de}

\smallskip
\noindent\textsc{Ilaria Del Corso} -- Dipartimento di Matematica, Università di Pisa, Largo Bruno Pontecorvo 5, 
56127 Pisa, Italy, \url{ilaria.delcorso@unipi.it}

\smallskip\noindent\textsc{Arnaud Plessis} -- Yanqi Lake Beijing Institute of Mathematical Sciences and Applications, Huairou District, Beijing, China, \url{plessis@bimsa.cn}

\smallskip
\noindent\textsc{Lea Terracini} -- Dipartimento di Informatica, Università di Torino, Corso Svizzera 185, 10149 Torino, Italy, \url{lea.terracini@unito.it}
\end{small}
\end{document}